\documentclass[preprint,1p,times]{elsarticle}
\usepackage{amsmath, amssymb, amsthm}
\usepackage{graphicx}
\usepackage{color,soul}
\usepackage{float}
\usepackage{lineno}
\usepackage{pgf}
\usepackage{epstopdf}

\journal{ }

\theoremstyle{definition}
\newtheorem{theorem}{Theorem}           
\newtheorem{lemma}{Lemma}               

\newtheorem{definition}{Definition}

\begin{document}

\begin{frontmatter}
	
\title{On the Dynamic Consistency of a Discrete Predator-Prey Model}
\author[cmbe]{Priyanka Saha\fnref{fn1}}
\author[cmbe]{Nandadulal Bairagi\corref{cor1}}
\ead{nbairagi.math@jadavpuruniversity.in}
\cortext[cor1]{Corresponding author}
\author[ajc]{Milan Biswas\fnref{fn2}}
\address[cmbe]{Centre for Mathematical Biology and Ecology\\ Department of Mathematics, Jadavpur University\\ Kolkata-700032, India.}
\address[ajc]{A. J. C. Bose College\\  A. J. C. Bose Road \\ Kolkata-700020, India.}

\begin{abstract}
We here discretize a predator-prey model by standard Euler forward method and non-standard finite difference method and then compare their dynamic properties with the corresponding continuous-time model. We show that NSFD model preserves positivity of solutions and is completely consistent with the dynamics of the corresponding continuous-time model. On the other hand, the discrete model formulated by forward Euler method does not show dynamic consistency with its continuous counterpart. Rather it shows scheme--dependent instability when step--size restriction is violated.
\end{abstract}

\end{frontmatter}

\section{Introduction}
Nonlinear system of differential equations play very important role in studying different physical, chemical and biological phenomena. However, in general, nonlinear differential equations cannot be solved analytically and therefore discretization is inevitable for good approximation of the solutions \cite{AL00}. Another reason of constructing discrete models, at least in case of population model, is that it permits arbitrary time-step units \cite{M89,M84}. Unfortunately, conventional discretization schemes, such as Euler method, Runge-Kutta method, show dynamic inconsistency \cite{M88}. It produces spurious solutions which are not observed in its parent model and its dynamics depend on the step-size. For example, consider the simple logistic model in continuous system:
\begin{eqnarray}\label{Continuous model0-int}
\frac{dx}{dt}=rx(1-\frac{x}{K}), ~~x(0)=x_0>0,
\end{eqnarray}
where $r$ and $K$ are positive constants. The system (\ref{Continuous model0-int}) has two equilibrium points with the following dynamical properties:
\begin{enumerate}
	\item the trivial equilibrium point $x=0$ is always unstable.
	\item the nontrivial equilibrium point $x=K$ is always stable.
\end{enumerate}
Fig. 1 shows that even if we start very close to zero ($x_0=0.3$) the solution goes to $x=K=50$, implying that the system is stable around the equilibrium point $x=K$ and unstable around $x=0$.\\
\begin{center}
	\includegraphics[width=3in, height=1.75in]{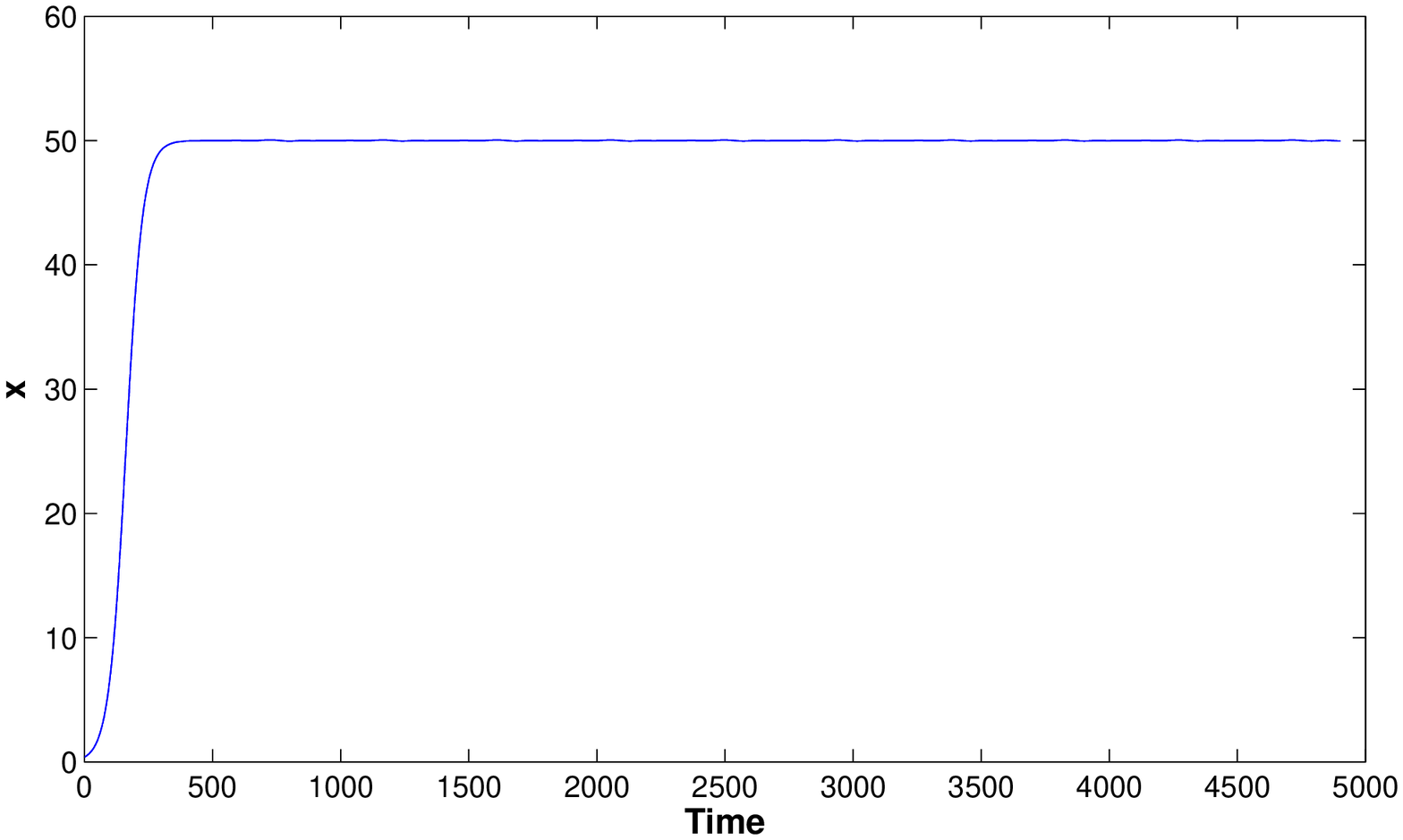}
\end{center}
{\bf Figure 1:} {\it Time series of the continuous system \eqref{Continuous model0-int}. It shows that the system (\ref{Continuous model0-int}) is stable around the interior equilibrium point $x=K$. Initial point and parameters are taken as $x(0)=0.4$, $r=3$ and $K=50$.}\\

The corresponding discrete model formulated by standard finite difference schemes (such as Euler forward method) is given by Anguelov and Lubuma\cite{AL00}
\begin{eqnarray}\label{Continuous model1-int}
\frac{x_{n+1}-x_n}{h}=rx_n(1-\frac{x_n}{K}).
\end{eqnarray}
This equation can be transformed into logistic difference equation
\begin{eqnarray}\label{discrete model-int1}
x_{n+1}=x_n+hrx_n(1-\frac{x_n}{K}),
\end{eqnarray}
where $h$ is the step-size.
The system \eqref{discrete model-int1} also has same equilibrium points with the following dynamic properties:
\begin{enumerate}
	\item the trivial equilibrium point $x=0$ is always unstable.
	\item the nontrivial equilibrium point $x=K$ is stable if $h<\frac{2}{r}$.
\end{enumerate}
The bifurcation diagram of Euler model \eqref{discrete model-int1} (Fig. 2) with $h$ as the bifurcating parameter shows that the fixed point $x=K$ changes its stability as the step-size $h$ crosses the value $\frac{2}{r}=0.666$. The fixed point is stable for $h<0.666$ and shows more complex behaviors (period doubling bifurcation) as the step-size is further increased. Thus, dynamics of Euler--forward model \eqref{discrete model-int1} depends on the step-size and exhibits spurious dynamics which are not observed in the corresponding continuous system \eqref{Continuous model0-int}.\\
\begin{center}
	\includegraphics[width=3in, height=1.75in]{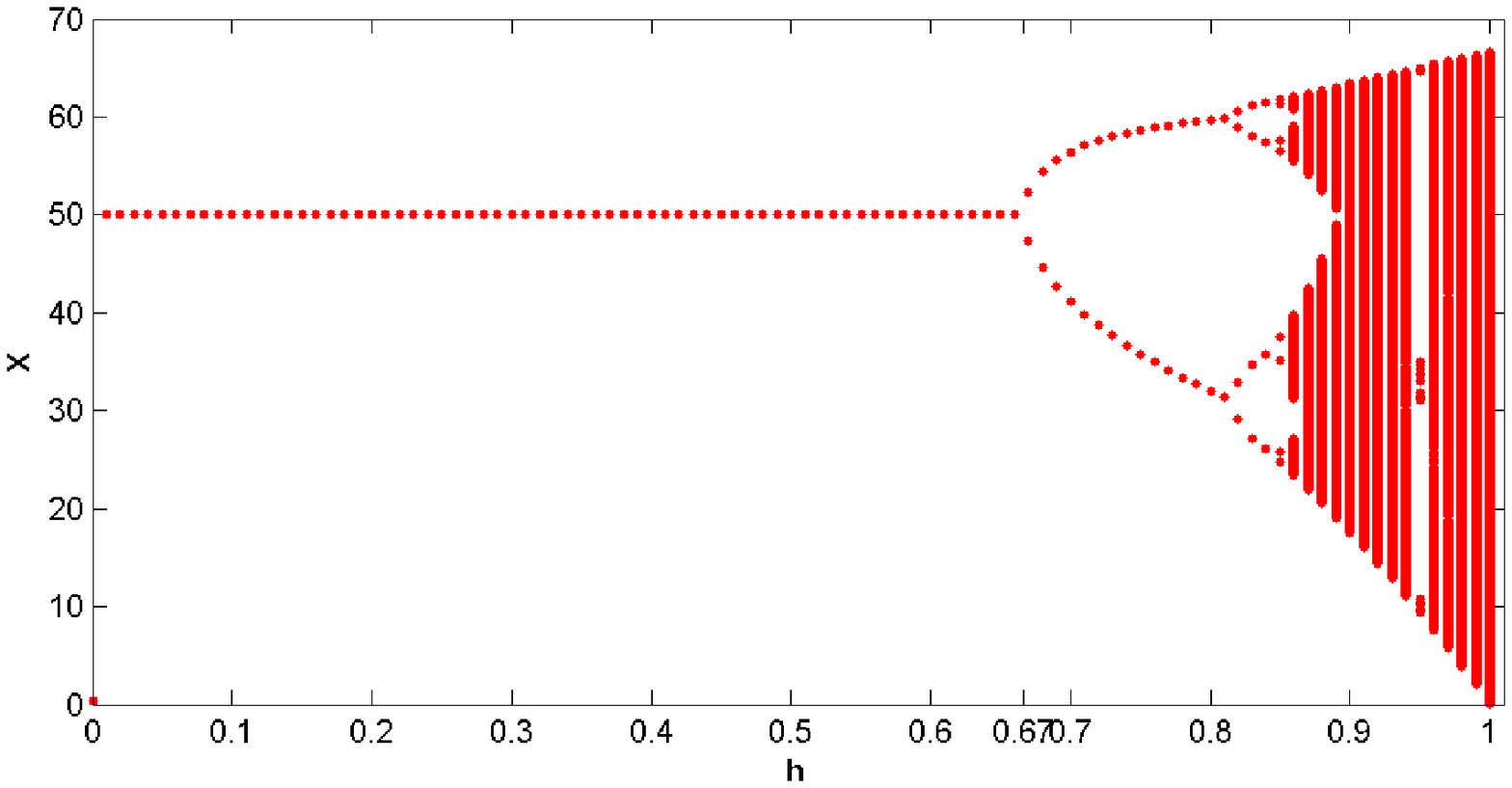}
\end{center}
\noindent\textbf{Figure 2:} {\it Bifurcation diagram of the model \eqref{discrete model-int1} with $h$ as the bifurcating parameter. It shows that the system is stable till the step-size $h$ is less than $0.666$ and unstable for higher values of $h$. Parameters and initial point are as in Fig. 1.}\\

Let us consider another simple example (decay equation)
\begin{eqnarray}\label{discrete model-int}
\frac{dx}{dt}=-\lambda x, ~ \lambda>0, x(0)=x_0>0.
\end{eqnarray}
Its solution, given by $$x(t)=x_0e^{-\lambda t},$$ is always positive. The corresponding discrete model constructed by Euler forward method is given by
\begin{eqnarray}\label{discrete model-int}
x_{n+1}=(1-\lambda h)x_n.
\end{eqnarray}
Note that its solution will not be positive if $\lambda h$ is sufficiently large and therefore supposed to show numerical instability.\\

These examples demonstrate that the discrete systems constructed by standard finite difference scheme is unable to preserve some properties of its corresponding continuous systems. Dynamic behaviors of the discrete model depend strongly on the step-size. However, on principles, the corresponding discrete system should have same properties to that of the original continuous system. It is therefore of immense importance to construct discrete model which will preserve the properties of its constituent continuous models. In the recent past, a considerable effort has been given in the construction of discrete-time model to preserve dynamic consistency of the corresponding continuous-time model without any limitation on the step-size. Mickens first proved that corresponding to any ODE, there exists an exact difference equation which has zero local truncation error \cite{M84,M88} and proposed a non-standard finite difference scheme (NSFD) in 1989 \cite{M89}. Later in 1994, he introduced the concept of elementary stability, the property which brings correspondence between the local stability at equilibria of the differential equation and the numerical method \cite{M94}. Anguelov and Lubuma \cite{AL01} formalized some of the foundations of Micken's rules, including convergence properties of non-standard finite difference schemes. They defined qualitative stability, which means that the constructed discrete system satisfies some properties like positivity of solutions, conservation laws and equilibria for any step-size. In 2005, Micken coined the term dynamic consistency, which means that a numerical method is qualitatively stable with respect to all desired properties of the solutions to the differential equation \cite{M05}. NSFD scheme has gained lot of attentions in the last few years because it generally does not show spurious behavior as compared to other standard finite difference methods. NSFD scheme has been successfully used in different fields like economics \cite{L13}, physiology \cite{SM03}, epidemic \cite{MET03,BB17a,SI10}, ecology \cite{RL13,BET17,G12,BB17} and physics \cite{M02,MO14}. Here we shall discretize a nonlinear continuous-time predator-prey system following dynamics preserving non-standard finite difference (NSFD) method introduced by Mickens \cite{M89}.\\

\noindent The paper is arranged in the following sequence. In the next section we describe the considered continuous-time model. Section 3 contains some definitions and general technique of constructing a NSFD model. Section 4 contains the analysis of NSFD and Euler models. Extensive simulations are presented in Section 5. The paper ends with the summary in Section 6.

\section{The model}
Celik \cite{C15} have investigated the following dimensionless Holling-Tanner predator-prey system with ratio-dependent functional response:
\begin{eqnarray}\label{model in continuous system}
\frac{dN}{dt} & = & N(1-N)-\frac{NP}{N+\alpha P},\\
\frac{dP}{dt} & = & \beta P(\delta -\frac{P}{N}).\nonumber
\end{eqnarray}
The state variables $N$ and $P$ represent, respectively, the density of prey and predator populations at time $t$,
and $N(t)>0,~P(t)\geq0$ for all $t$. Here $\alpha$, $\beta$ and $\delta$ are positive constants. For more description of the
model, readers are referred to the work of Celik \cite{C15}. \\

Celik \cite{C15} discussed about the existence and stability of the coexistence interior equilibrium $E^*=(N^*,P^*)$, where
\begin{eqnarray}
N^*=\frac{1+\alpha \delta-\delta}{1+\alpha \delta}, \nonumber
~~~~~P^*=\delta N^*.
\end{eqnarray}
The following results are known.\\

\noindent\textbf{Theorem 1.1.} {\it The interior equilibrium point $E^*$ of the system \eqref{model in continuous system} exists and becomes stable if
	\begin{eqnarray}\label{stability condition}
	(i) \alpha \delta +1>\delta, ~(ii) \delta (2+\alpha \delta)<(1+\alpha \delta)^2 (1+\beta \delta).\nonumber
	\end{eqnarray}}
Here we seek to construct a discrete model of the corresponding continuous model \eqref{model in continuous system} that preserves the qualitative properties of the continuous system and maintains dynamic consistency. We also construct the corresponding Euler discrete model and compare its results with the results of NSFD model.


\section{Some definitions}
Consider the differential equation
\begin{eqnarray}\label{Continuous model}
\frac{dx}{dt}=f(x,t,\lambda),
\end{eqnarray}
where $\lambda$ represents the parameter defining the system \eqref{Continuous model}. Assume that a finite difference scheme corresponding to the continuous system \eqref{Continuous model} is described by
\begin{eqnarray}\label{Discrete model}
x_{k+1}=F(x_{k},t_{k},h,\lambda).
\end{eqnarray}
We assume that $F(., ., ., .)$ is such that the proper uniqueness--existence properties holds; the step size is $h=\nabla t$ with $t_k=hk$, $k=$ integer; and $x_k$ is an approximation to $x(t_k)$.
\begin{definition} \cite{M05}\label{definition1} ~Let the differential equation \eqref{Continuous model} and/or its solutions have a property $P$. The discrete model \eqref{Discrete model} is said to be dynamically consistent with the equation \eqref{Continuous model} if it and/or its solutions also have the property $P$.
\end{definition}

\begin{definition} \cite{M05,DK05,AL03}\label{definition2}
	The NSFD procedures are based on just two fundamental rules:
	
	$~~~$(i) the discrete first--derivative has the representation
	
	$~~~~~~~~~~~~~~~$$\frac{dx}{dt} \rightarrow \frac{x_{k+1}-\psi(h)x_k}{\phi(h)}$, $h=\triangle t$,\\
	where $\phi(h)$, $\psi(h)$ satisfy the conditions
	$\psi(h)=1+O(h^2)$,~ $\phi(h)=h+O(h^2)$;
	
	$~~~$(ii) both linear and nonlinear terms may require a nonlocal representation on the discrete computational lattice; for example,
	
	$~~~~~~~~~~~~~~~$$x\rightarrow 2x_k-x_{k+1}$,~~~~ $x^3\rightarrow (\frac{x_{k+1}+x_{k-1}}{2})x_k^2$,
	
	$~~~~~~~~~~~~~~~$$x^3\rightarrow 2x_k^3-x_k^2x_{k+1}$, ~~~~$x^2\rightarrow (\frac{x_{k+1}+x_k+x_{k-1}}{3})x_k$.
	
	\noindent While no general principles currently exist for selecting the functions $\psi(h)$ and $\phi(h)$, particular forms for a specific equation can easily be determined. Functional forms commonly used for $\psi(h)$ and $\phi(h)$ are
	$$\phi(h)=\frac{1-e^{-\lambda h}}{\lambda}, ~\psi(h)=cos(\lambda h),$$
	where $\lambda$ is some parameter appearing in the differential equation.
\end{definition}
\begin{definition}\label{definition3}
	The finite difference method \eqref{Discrete model} is called positive if for any value of the step size $h$, solution of the discrete system remains positive for all positive initial values.
\end{definition}
\begin{definition}\label{definition4}
	The finite difference method \eqref{Discrete model} is called elementary stable if for any value of the step size $h$, the fixed points of the difference equation are those of the differential system and the linear stability properties of each fixed point being the same for both the differential system and the discrete system.
\end{definition}
\begin{definition} \cite{DK06}\label{definition5}
	A method that follows the Mickens rules (given in the Definition 3.2) and preserves the positivity of the solutions is called positive and elementary stable nonstandard (PESN) method.
\end{definition}

\section{Nonstandard finite difference (NSFD) model}
For convenience, at first we can write the continuous system \eqref{model in continuous system} as
\begin{eqnarray}\label{model in continuous system 2}
\frac{dN}{dt} & = & N-N^2-\frac{NP}{(N+\alpha P)}+(N-N)(N+\alpha P),\\
\frac{dP}{dt} & = & \beta \delta P-\frac{\beta P^2}{N}.\nonumber
\end{eqnarray}\\
Now we express the above system as follows:
\begin{eqnarray}\label{continuous form}
\frac{dN}{dt} & = & N-N^2-NA(N,P)+(N-N)B(N,P),\\
\frac{dP}{dt} & = & \beta \delta P-\beta PC(N,P),\nonumber
\end{eqnarray}\\
where $A(N,P)=\frac{P}{N+\alpha P}$, $B(N,P)=(N+\alpha P)$ and $C(N,P)=\frac{P}{N}$.\\
We employ the following non-local approximations termwise for the system \eqref{continuous form}:
\begin{eqnarray}\label{Nonlocal approxmiation}
\left\{
\begin{array}{ll}
\frac{dN}{dt}\rightarrow\frac{N_{n+1}-N_n}{h},~~~~~~~~~~~~~~~~~~\frac{dP}{dt}\rightarrow\frac{P_{n+1}-P_n}{h}\\
N\rightarrow N_n,~~~~~~~~~~~~~~~~~~~~~~~~~~~~P\rightarrow P_n,\\
N^2\rightarrow N_nN_{n+1},\\
PC(N,P)\rightarrow P_{n+1}C(N_n,P_n),\\
NA(N,P)\rightarrow N_{n+1}A(N_n,P_n),\\
(N-N)B(N,P)\rightarrow (N_n-N_{n+1})B(N_n,P_n),\\
\end{array}
\right.
\end{eqnarray}
where $h~(>0)$ is the step-size.\\
By these transformations, the continuous-time system \eqref{model in continuous system 2} is converted to
\begin{eqnarray}\label{discrete system}
\frac{N_{n+1}-N_n}{h}&=&N_n-N_n N_{n+1}-\frac{N_{n+1}P_n}{N_n+\alpha P_n}+(N_n-N_{n+1})(N_n+\alpha P_n),\nonumber \\
\frac{P_{n+1}-P_n}{h}&=&\beta \delta P_n-\frac{\beta P_{n+1}P_n }{N_n}.
\end{eqnarray}
System \eqref{discrete system} can be simplified to
\begin{eqnarray}\label{model in discrete system}
N_{n+1} & = & \frac{N_n\{1+h+h(N_n+\alpha P_n)\}(N_n+\alpha P_n)}{(1+2hN_n+\alpha hP_n)(N_n+\alpha P_n)+hP_n},\\
P_{n+1} & = & \frac{P_nN_n(1+\beta \delta h)}{N_n+\beta hP_n}.\nonumber
\end{eqnarray}
Note that all solutions of the discrete-time system \eqref{model in discrete system} remains positive for any step-size if they start with positive initial values. Therefore, the system \eqref{model in discrete system} is positive.
\subsection{Existence of fixed points}
Fixed points of the system \eqref{model in discrete system} are the solutions of the coupled algebraic equations obtained by putting $N_{n+1}=N_n=N$ and $P_{n+1}=P_n=P$ in \eqref{model in discrete system}. However, the fixed points can be obtained more easily from \eqref{discrete system} with the same  substitutions. Thus, fixed points are the solutions of the following nonlinear algebraic equations:
\begin{eqnarray}\label{existence1}
N-N^2-\frac{NP}{N+\alpha P}=0,\\
\beta \delta P-\frac{\beta P^2}{N}=0.\nonumber
\end{eqnarray}
It is easy to observe that $E_1=(1,0)$ is the predator-free fixed point. The interior fixed point $E^*=(N^*,P^*)$ satisfies
\begin{eqnarray}\label{existence2}
1-N^*-\frac{P^*}{N^*+\alpha P^*}=0 ~and~ \delta-\frac{P^*}{N^*}=0.
\end{eqnarray}
From the second equation of \eqref{existence2}, we have $P^*=\delta N^*$. Substituting $P^*$ in the first equation of \eqref{existence2}, we find $N^*=\frac{1+\alpha \delta-\delta}{1+\alpha \delta}$, which is always positive if $1+\alpha \delta >\delta$. Thus  the positive fixed point $E^*$ exists if $1+\alpha \delta>\delta$.
\subsection{Stability analysis of fixed points}
The variational matrix of system \eqref{model in discrete system} evaluated at an arbitrary fixed point $(N,P)$ is given by\\
\begin{equation} \label{jacobian1}J(N,P)=\left(
\begin{array}{cc}
a_{11} ~~ a_{12}\\
a_{21} ~~ a_{22}\\
\end{array}
\right),\end{equation}\\
where
\begin{equation}\label{jacobian}
\left\{
\begin{array}{ll}
a_{11} = \frac{\{1+h+h(N_n+\alpha P_n)\}(N_n+\alpha P_n)}{(1+2hN_n+\alpha hP_n)(N_n+\alpha P_n)+hP_n}+\frac{hN_n(N_n+\alpha P_n)}{(1+2hN_n+\alpha hP_n)(N_n+\alpha P_n)+hP_n}\\
~~~~~~~~+\frac{N_n\{1+h+h(N_n+\alpha P_n)\}}{(1+2hN_n+\alpha hP_n)(N_n+\alpha P_n)+hP_n}\\
~~~~~~~~-\frac{N_n\{1+h+h(N_n+\alpha P_n)\}(N_n+\alpha P_n)\{2h(N_n+\alpha P_n)+(1+2hN_n+\alpha hP_n)\}}{\{(1+2hN_n+\alpha h P_n)(N_n+\alpha P_n)+hP_n\}^2},\\
\\
a_{12}=\frac{\alpha hN_n(N_n+\alpha P_n)}{(1+2hN_n+\alpha hP_n)(N_n+\alpha P_n)+hP_n}+\frac{\alpha N_n\{1+h+h(N_n+\alpha P_n)\}}{(1+2hN_n+\alpha hP_n)(N_n+\alpha P_n)+hP_n}\\
~~~~~~~~ -\frac{N_n\{1+h+h(N_n+\alpha P_n)\}(N_n+\alpha P_n)\{\alpha h(N_n+\alpha P_n)+\alpha (1+2hN_n+\alpha hP_n)+h\}}{\{(1+2hN_n+\alpha hP_n)(N_n+\alpha P_n)+hP_n\}^2},\\
\\
a_{21}=\frac{P_n(1+\beta \delta h)}{N_n+\beta hP_n}-\frac{P_nN_n(1+\beta \delta h)}{(N_n+\beta hP_n)^2},\\
\\
a_{22}=\frac{(1+\beta \delta h)N_n}{N_n+\beta hP_n}-\frac{\beta hP_nN_n(1+\beta \delta h)}{(N_n+\beta hP_n)^2}.\nonumber
\end{array}
\right.
\end{equation}
Let $\lambda_{1}$ and $\lambda_{2}$ be the eigenvalues of the variational matrix \eqref{jacobian1} then we give the following definition in relation to the stability of the system \eqref{model in discrete system}.
\begin{definition}\label{definition6}
	A fixed point $(x,y)$ of the system \eqref{model in discrete system} is called stable if $\left|\lambda_{1}\right|<1$, $\left|\lambda_{2}\right|<1$ and a source if $\left|\lambda_{1}\right|>1$, $\left|\lambda_{2}\right|>1$. It is called a saddle if $\left|\lambda_{1}\right|<1$, $\left|\lambda_{2}\right|>1$ or $\left|\lambda_{1}\right|>1$, $\left|\lambda_{2}\right|<1$ and a non--hyperbolic fixed point if either $\left|\lambda_{1}\right|=1$ or $\left|\lambda_{2}\right|=1$.
\end{definition}
\begin{lemma} \cite{LG13}\label{lemma}
	Let $\lambda_{1}$ and $\lambda_{2}$ be the eigenvalues of the variational matrix \eqref{jacobian1}.
	Then $\left|\lambda_{1}\right|<1$ and $\left|\lambda_{2}\right|<1$ iff
	$(i) 1-det(J)>0, (ii) 1-trace(J)+det(J)>0$ and $(iii) 0<a_{11}<1,~ 0<a_{22}<1.$
\end{lemma}

%

\begin{theorem} {\it Suppose that conditions of Theorem 1.1 hold. Then the fixed point $E^*$ of the system \eqref{model in discrete system} is locally asymptotically stable.}\end{theorem}

\noindent\textbf{Proof.} At the interior fixed point $E^*$, the variational matrix reads as\\
$$J(N^*,P^*)=\left(
\begin{array}{cc}
a_{11}^* ~~ a_{12}^*\\
a_{21}^* ~~ a_{22}^*
\end{array}
\right),$$\\
where
\begin{eqnarray}\label{interior jacobian}
\left\{
\begin{array}{ll}
a_{11}^*=1+\frac{N^*h(1-2N^*-\alpha P^*)}{G},\\
a_{12}^*=\frac{N^*h(\alpha-\alpha N^*-1)}{G},\\
a_{21}^*=\frac{\beta \delta hP^*}{H},\\
a_{22}^*=1-\frac{\beta hP^*}{H}
\end{array}
\right.
\end{eqnarray}\\
with $G=\{1+h+h(N^*+\alpha P^*)\}(N^*+\alpha P^*)$ and $H=(1+\beta \delta h)N^*$.\\
Using $P^*=\delta N^*$ in \eqref{interior jacobian}, we have
\begin{eqnarray}\label{interior jacobian 1}
\left\{
\begin{array}{ll}
a_{11}^*=1+\frac{N^*h(1-2N^*-\alpha \delta N^*)}{G},\\
a_{12}^*=\frac{N^*h(\alpha-\alpha N^*-1)}{G},\\
a_{21}^*=\frac{\beta \delta^2 hN^*}{H},\\
a_{22}^*=1-\frac{\beta \delta h N^*}{H}.
\end{array}
\right.
\end{eqnarray}
One can compute that $1-det(J)=-\frac{(N^*)^2h\{(1-\beta \delta-\alpha \beta \delta^2)-(2+\alpha \delta)N^*\}}{GH}+\frac{\beta \delta h^2(N^*)^2\{N^*(1+\alpha \delta)+N^*(1+\alpha \delta)^2+\frac{\alpha \delta^2}{1+\alpha \delta}\}}{GH}>0$, provided $-(1-\beta \delta-\alpha \beta \delta^2)+(2+\alpha \delta)N^*>0$, i.e. $\delta (2+\alpha \delta)<(1+\alpha \delta)^2(1+\beta \delta)$.
Note that $trace(J)=\frac{(N^*)^2}{GH}[(1+\alpha \delta)\{2+h(2+\beta\delta+2N^*\alpha \delta)\}+h(1+N^*\alpha \delta)+h^2\beta \delta\{\frac{2\delta}{1+\alpha \delta}+\alpha \delta(1+N^*+N^*\alpha \delta)+N^*\}]>0$ and $1-trace(J)+det(J)=\frac{\beta \delta h^2 (N^*)^2(1+\alpha \delta-\delta)}{GH}>0$, following the existing condition of $E^*$. Therefore, the positive fixed point $E^*$ is locally asymptotically stable provided conditions of Theorem 1.1 hold. Hence the theorem is proven.
\subsection{The Euler forward method}
By Euler's forward method, we transform the continuous model \eqref{model in continuous system} in the following discrete model:
\begin{eqnarray} \label{model in Euler system}
\frac{N_{n+1}-N_{n}}{h}&=&N_n[1-N_n-\frac{P_n}{N_n+\alpha P_n}],\\
\frac {P_{n+1}-P_{n}}{h}&=&\beta P_n[\delta-\frac{P_n}{P_n}],\nonumber
\end{eqnarray}
where $h>0$ is the step size.
Rearranging the above equations, we have
\begin{eqnarray} \label{model in Euler system1}
N_{n+1}&=&N_{n}+h N_n[1-N_n-\frac{P_n}{N_n+\alpha P_n}],\\
P_{n+1}&=&P_{n}+h\beta P_n[\delta-\frac{P_n}{N_n}]. \nonumber
\end{eqnarray}
It is to be noticed that the system \eqref{model in Euler system1} with positive initial values is not unconditionally positive due to the presence of negative terms. The system may therefore exhibit spurious behaviors and numerical instabilities \cite{M94}.
\subsubsection{Existence and stability of fixed points}
At the fixed point, we substitute $N_{n+1}=N_{n}=N$ and $P_{n+1}=P_{n}=P$. One can easily compute that \eqref{model in Euler system1} has the same interior fixed points as in the previous case. The fixed point $E_1=(1,0)$ always exist and the fixed point $E^{*}=(N^{*},P^{*})$ exists if $1+\alpha \delta>\delta$, where $N^*=\frac{1+\alpha \delta-\delta}{1+\alpha \delta}$, $P^* = \delta N^*$. We are interested for interior equilibrium only.\\
The variational matrix of the system \eqref{model in Euler system1} at any arbitrary fixed point $(N, P)$ is given by
$$J(x,y)=\left(
\begin{array}{cc}
a_{11} & a_{12} \\
a_{21} & a_{22}  \\
\end{array}
\right),$$

\begin{eqnarray}\label{Jacobian Matrix Elements}
\nonumber
\mbox{where}
\left\{
\begin{array}{ll}
a_{11}=1+h[1-N_n-\frac{P_n}{N_n+\alpha P_n}]+h N_n[-1+\frac{P_n}{(N_n+\alpha P_n)^2}],\\
a_{12}=-h (\frac{N_n}{N_n+\alpha P_n})^2,\\
a_{21}=h \beta (\frac{P_n}{N_n})^2,\\
a_{22}=1+h[\beta \delta-\frac{\beta P_n}{N_n}-\beta \frac{P_n}{N_n}].
\end{array}
\right.
\end{eqnarray}

\begin{theorem}
	 Suppose that the conditions of Theorem 1.1 hold. The interior fixed point $E^*$ of the system \eqref{model in Euler system1} is then locally asymptotically stable if $h<min[\frac{G}{H},\frac{2(1+\alpha \delta)^2}{G}]$, where $G=(1+\alpha \delta)^2 (1+\beta \delta)-\delta (2+\alpha \delta)$, $H=\beta \delta (1+\alpha \delta-\delta)(1+\alpha \delta).$
\end{theorem}

\begin{proof} At the interior equilibrium point $E^*$, the Jacobian matrix is evaluated as $$J(N^{*},P^{*})=\left(
	\begin{array}{cc}
	a_{11} & a_{12} \\
	a_{21} & a_{22}  \\
	\end{array}
	\right),$$
	where $a_{11}=1-hN^*[1-\frac{P^*}{(N^*+\alpha P^*)^2}]$, $a_{12}=-h(\frac{N^*}{N^*+\alpha P^*})^2$, $a_{21}=h\beta (\frac{P^*}{N^*})^2$, $a_{22}=1-h\beta \frac{P^*}{N^*}$. Note that $1-trace(J)+det(J) =h^2 \beta P^*$ is always positive, following the existence conditions of $E^*$. Thus, condition (ii) of Lemma 4.1 is satisfied. One can compute that $det(J)=1-h N^*[\frac{G}{H}-h]$. Here $H$ is positive following the existence condition of $E^*$ and $G>0$ if $(1+\alpha \delta)^2 (1+\beta \delta)>\delta (2+\alpha \delta)$. Thus condition (i) of Lemma 4.1 is satisfied if $h>\frac{G}{H}$. Simple computations give $1+trace(J)+det(J)=2(2-h\frac{G}{(1+\alpha \delta)^2})+h^2H$. This expression will be positive if $0<h<\frac{2(1+\alpha \delta)^2}{G}$. Therefore, coexistence equilibrium point $E^*$ exists and becomes stable if $1+\alpha \delta>\delta$, $\delta (2+\alpha \delta)<(1+\alpha \delta)^2 (1+\beta \delta)$ and $h<min[\frac{G}{H},\frac{2(1+\alpha \delta)^2}{G}]$. Hence the theorem.
\end{proof}
\noindent {\bf Remark 4.1.} Note that if $h>\frac{G}{H}$ then $E^*$ is unstable even when the other two conditions are satisfied.
\section{Numerical simulations}
In this section, we present some numerical simulations to validate our analytical results of the NSFD discrete system \eqref{model in discrete system} and the Euler system \eqref{model in Euler system1} with their continuous counterpart \eqref{model in continuous system}. For this experiment, we consider the parameters set as in Celik \cite{C15}: $\alpha=0.7, \beta=0.9, \delta=0.6$.
The step size is kept fixed at $h=0.1$ in all simulations, if not stated otherwise. We consider the initial value $I_1=(0.2, 0.2)$ as in Celik \cite{C15} for all simulations. For the above parameter set, the interior fixed point is evaluated as $E^*=(N^*, P^*)= (0.5775, 0.3465)$. We first reproduce the phase plane diagrams (Fig. 3) of the continuous system \eqref{model in continuous system}, the NSFD discrete system \eqref{model in discrete system} and the Euler discrete system \eqref{model in Euler system1} by using ODE45 of the software Matlab 7.11. Following the analytical results stated in the section 3, the phase plane diagrams show that the equilibrium $E^*$ is stable for all three cases.
\begin{center}
	\includegraphics[width=2in, height=1.5in]{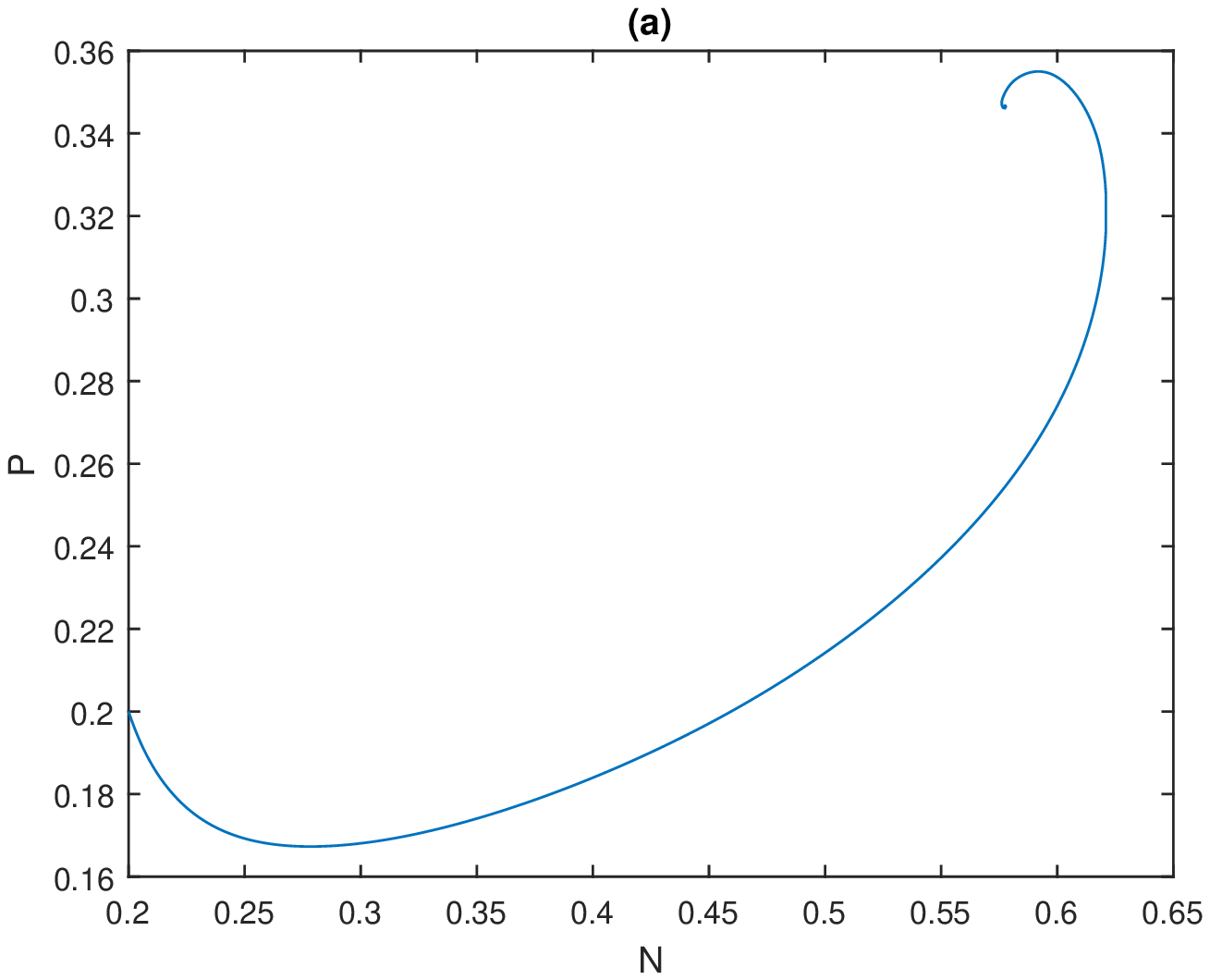}
	\includegraphics[width=2in, height=1.5in]{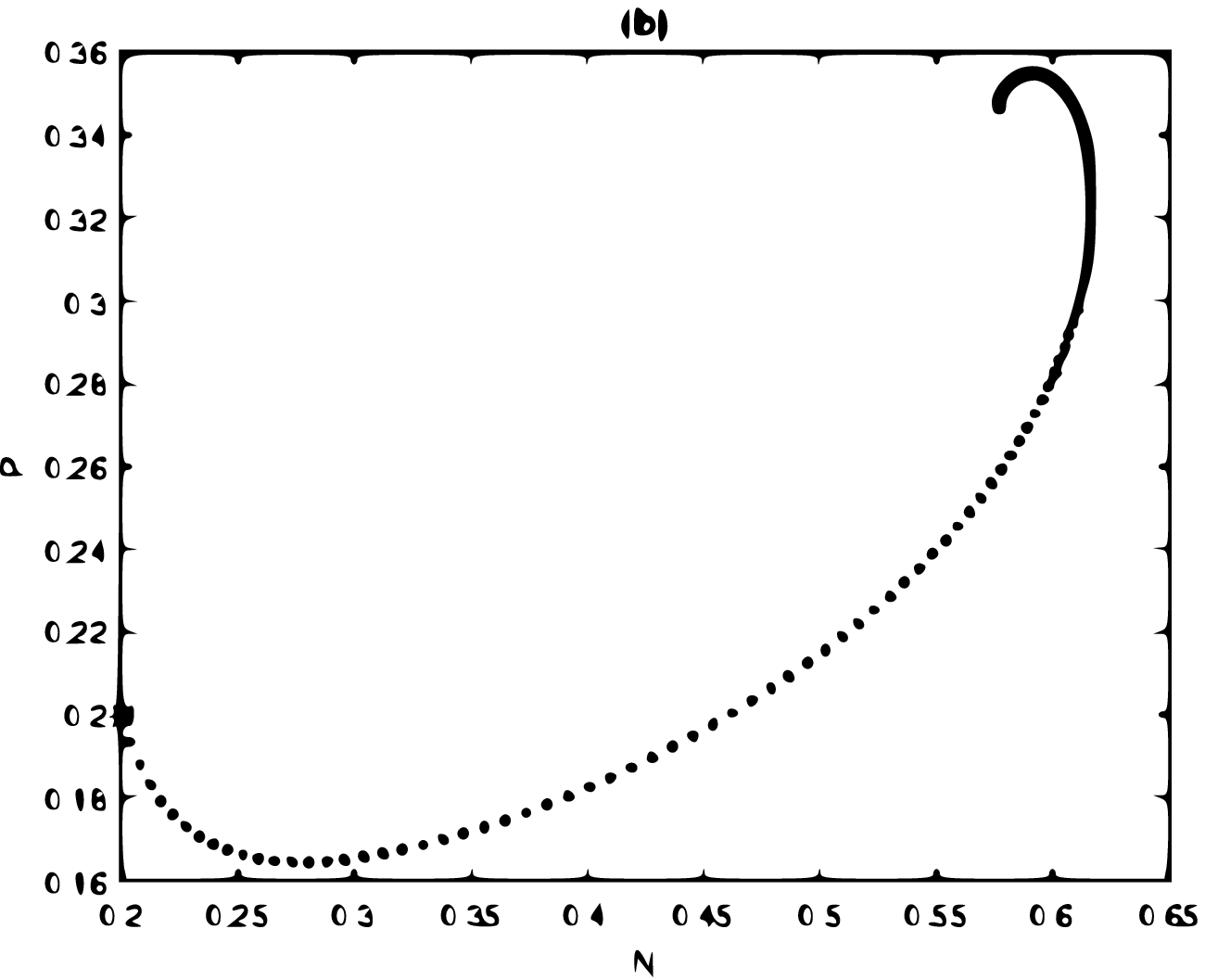}
	\includegraphics[width=2in, height=1.5in]{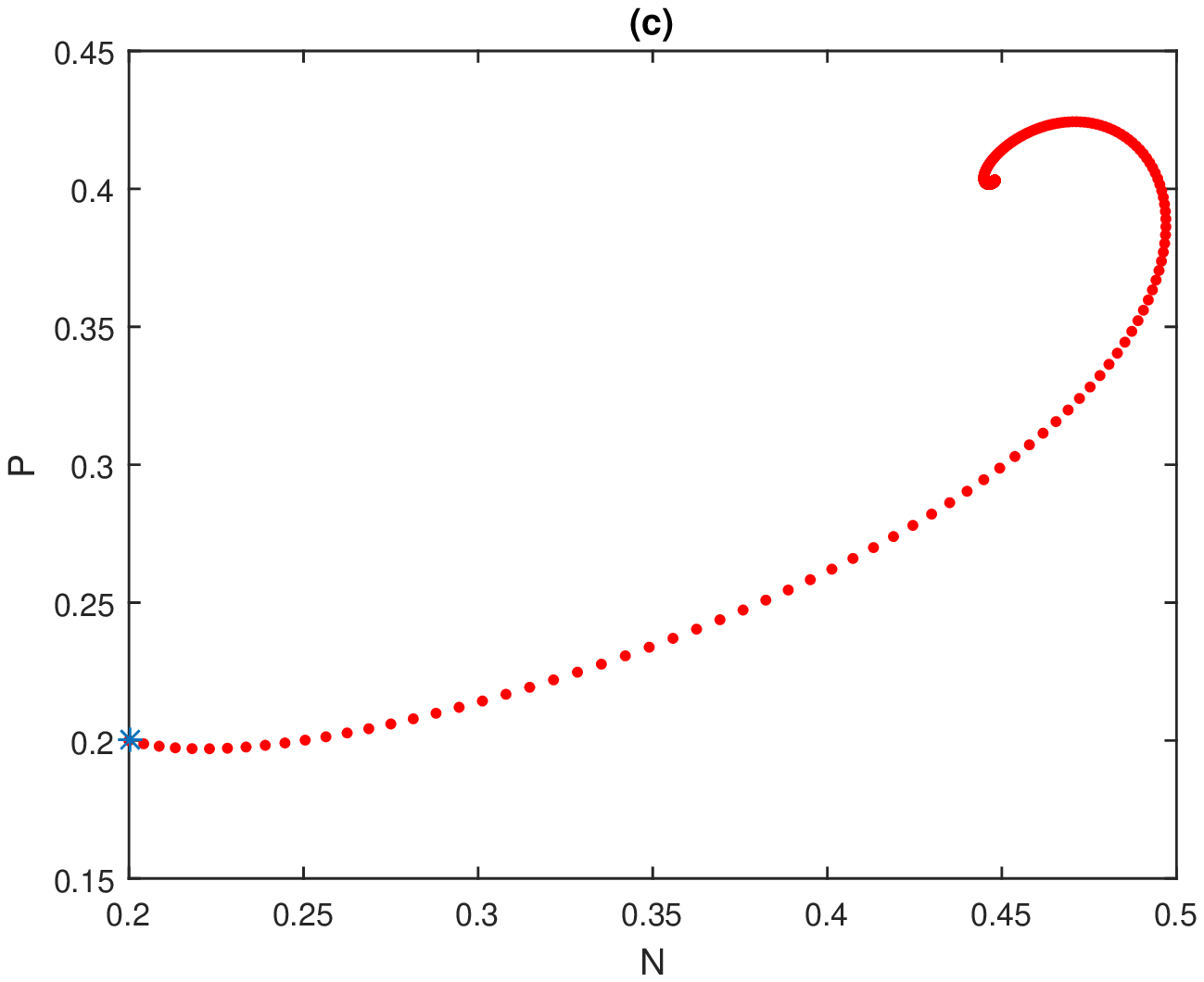}
\end{center}
{\bf Figure 3:} {\it Phase diagrams of the continuous system \eqref{model in continuous system} (Fig. a), the NSFD discrete system \eqref{model in discrete system} (Fig. b) and the Euler system \eqref{model in Euler system1} (Fig. c). These figures show that solution in each case converges to the stable coexistence equilibrium $E^*$ for the parameters $\alpha=0.7, \beta=0.9, \delta=0.6$. Here $G=(1+\alpha \delta)^2 (1+\beta \delta)-\delta (2+\alpha \delta)=1.6533$ and $h=0.1<min\{\frac{G}{H},\frac{2(1+\alpha \delta)^2}{G}\}=min\{2.6293, 2.4393\}$.}

\begin{center}
	\includegraphics[width=2in, height=1.5in]{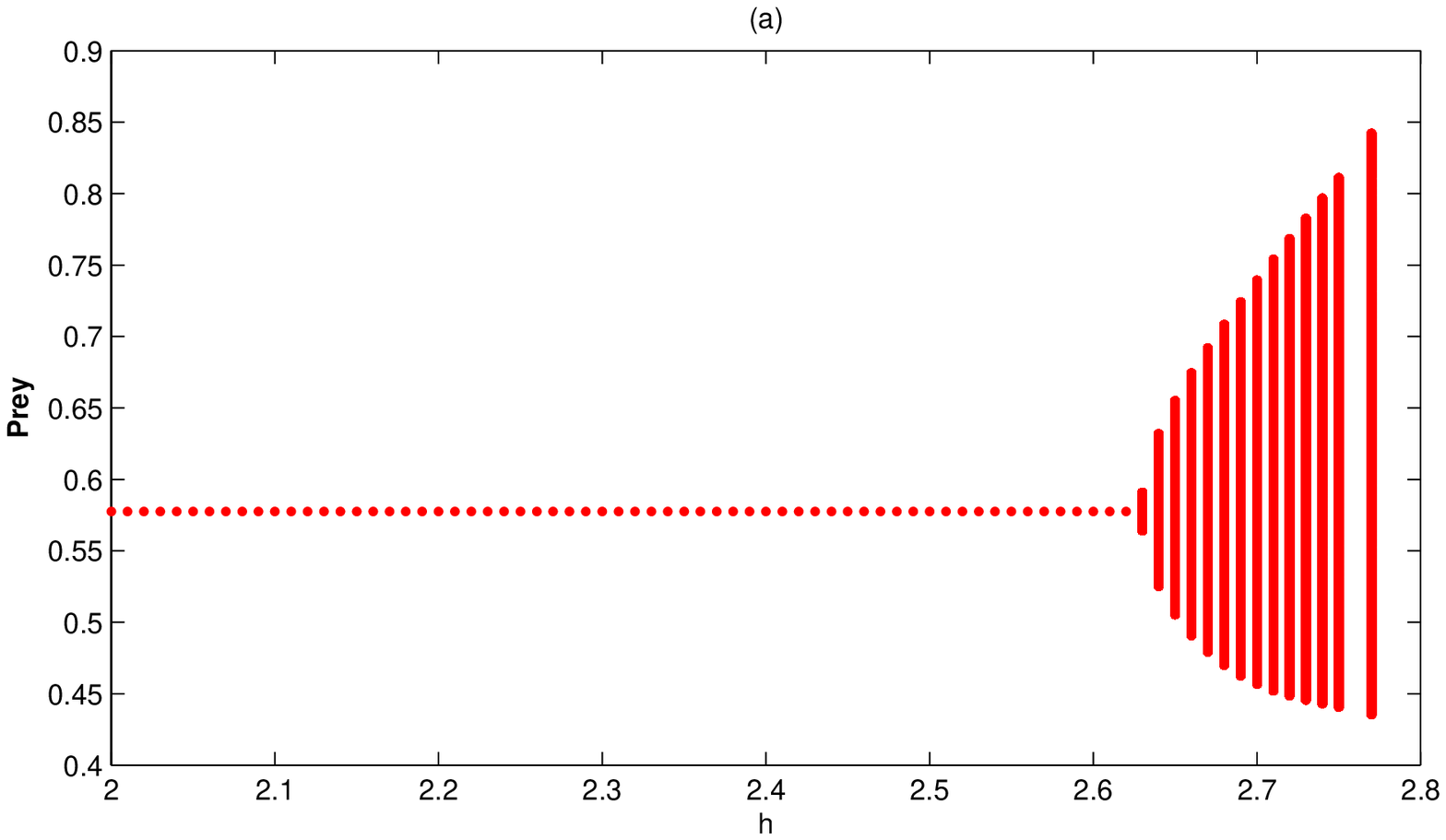}
	\includegraphics[width=2in, height=1.5in]{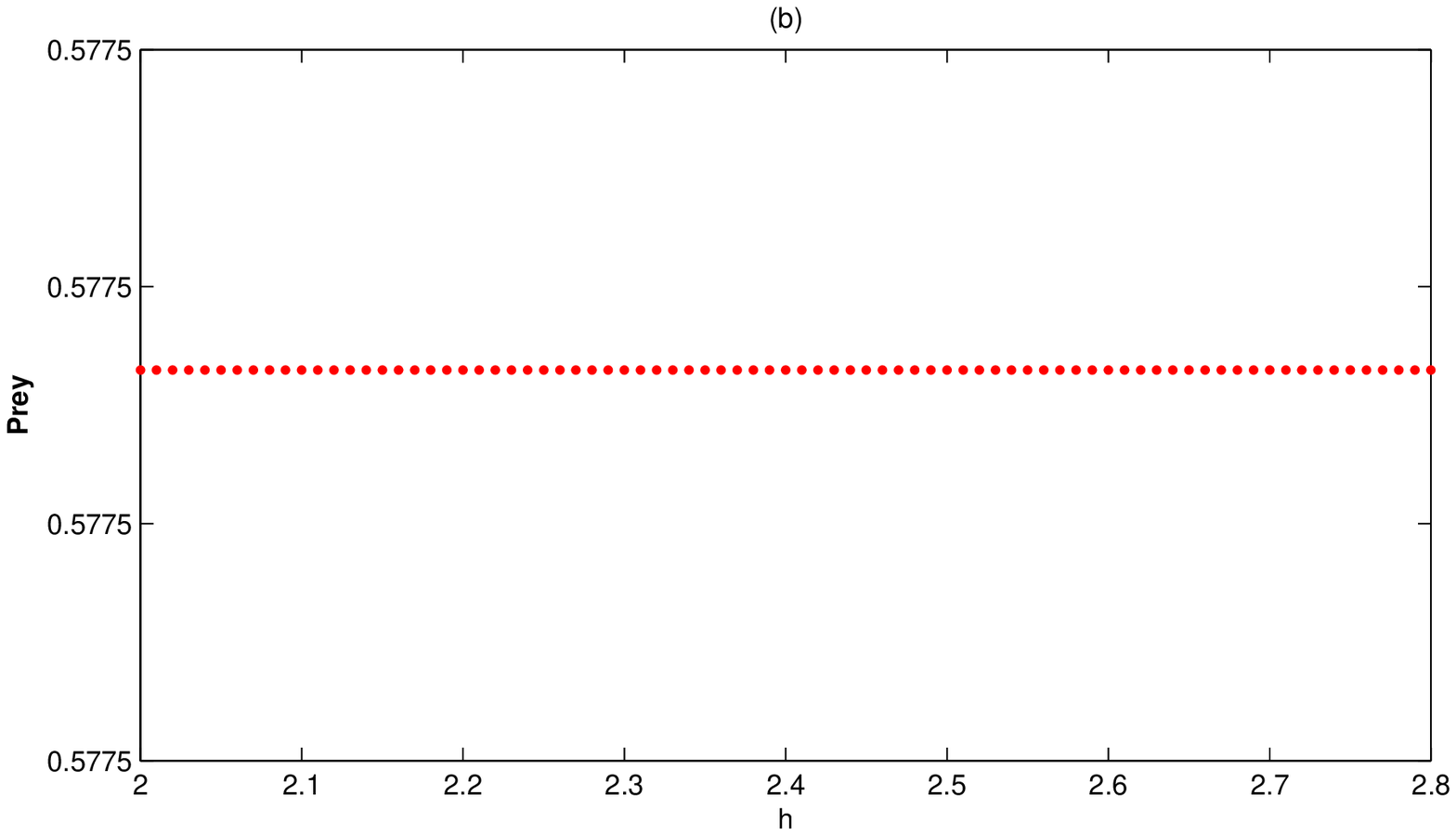}
\end{center}
{\bf Figure 4:} {\it Bifurcation diagrams of prey population of Euler--forward model \eqref{model in Euler system1} (Fig. a) and NSFD model \eqref{model in discrete system} (Fig. b) with step--size $h$ as the bifurcation parameter. All the parameters and initial value are same as in Fig. 3. The first figure shows that the prey population is stable for small step--size $h$ and unstable for higher value of $h$. The second figure shows that the prey population is stable for all step--size $h$.}\\

To compare step--size dependency of the Euler model and NSFD model, we have plotted the bifurcation diagrams of prey population of the systems \eqref{model in Euler system1} and \eqref{model in discrete system} considering step--size $h$ as the bifurcation parameter (Fig. 4) for the same parameter values as in Fig. 3. Fig. 4a shows that behavior of the Euler model depends on the step--size. If step--size is small, system population is stable and the dynamics resembles with the continuous system \eqref{model in continuous system}. As the step--size is increased, system population becomes unstable and therefore the dynamics is inconsistent with the continuous system. However, the second figure (Fig. 4b) shows that the NSFD model \eqref{model in discrete system} remains stable for all $h$, indicating that the dynamics is independent of step--size.
\begin{center}
	\includegraphics[width=2in, height=1.5in]{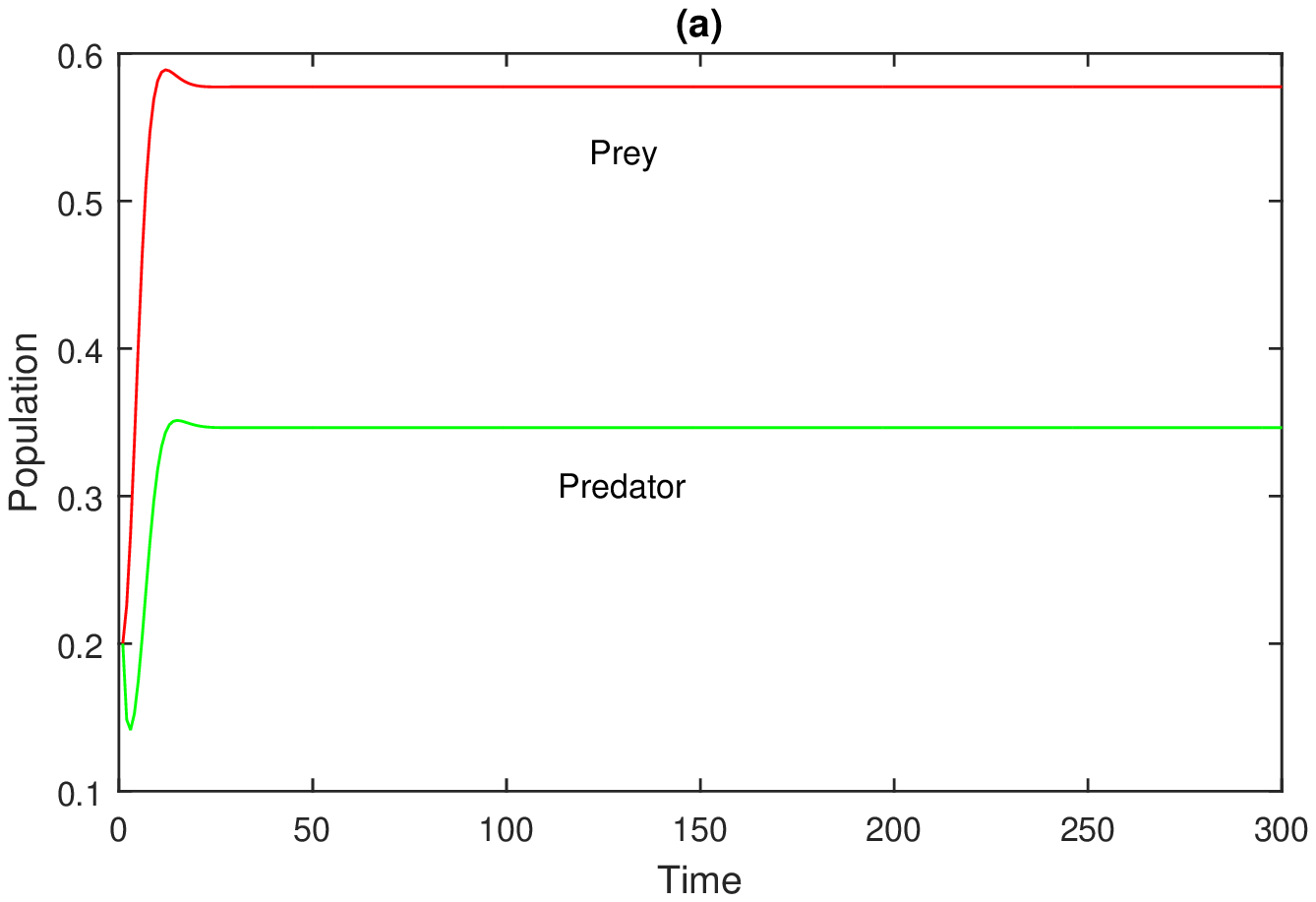}
	\includegraphics[width=2in, height=1.5in]{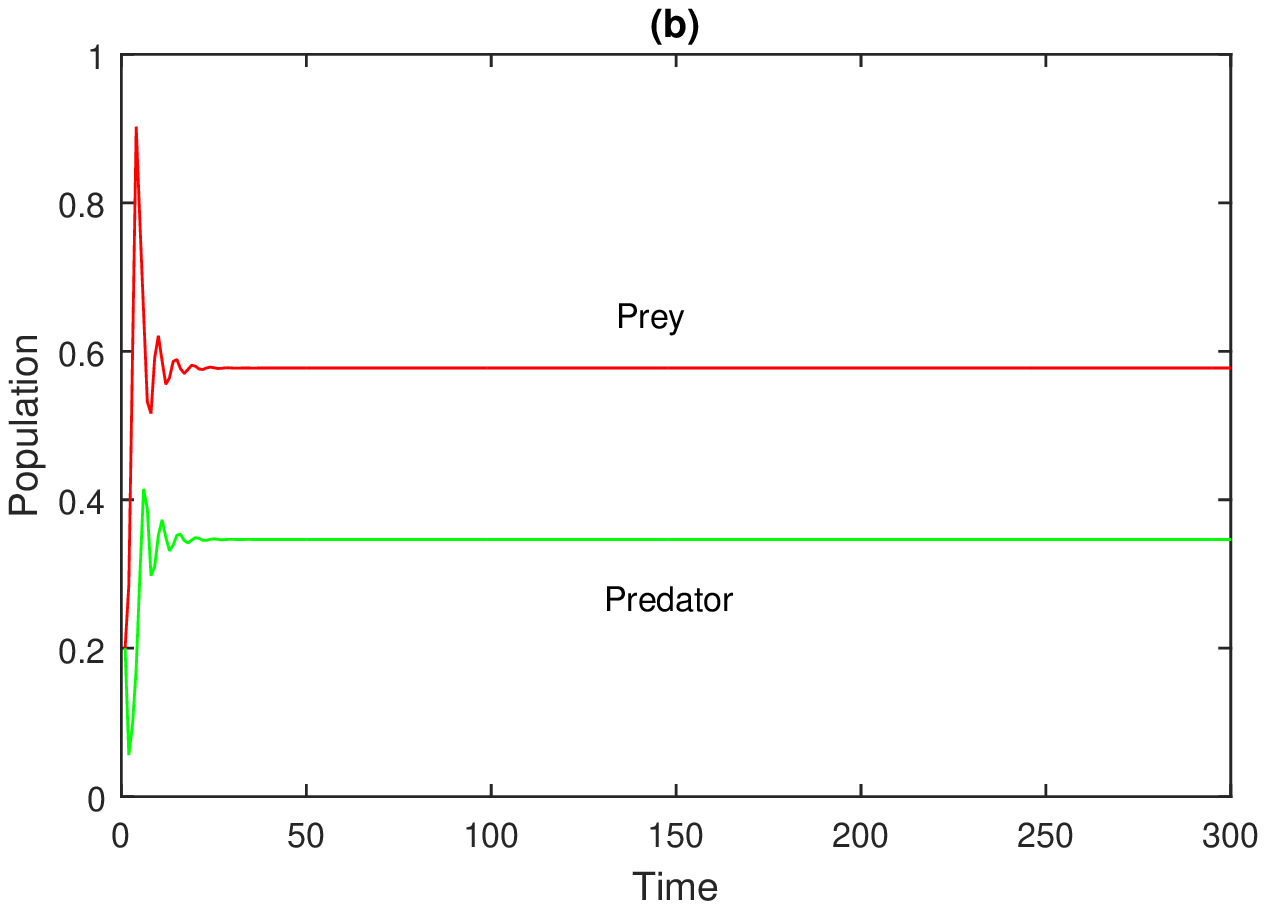}
	\includegraphics[width=2in, height=1.5in]{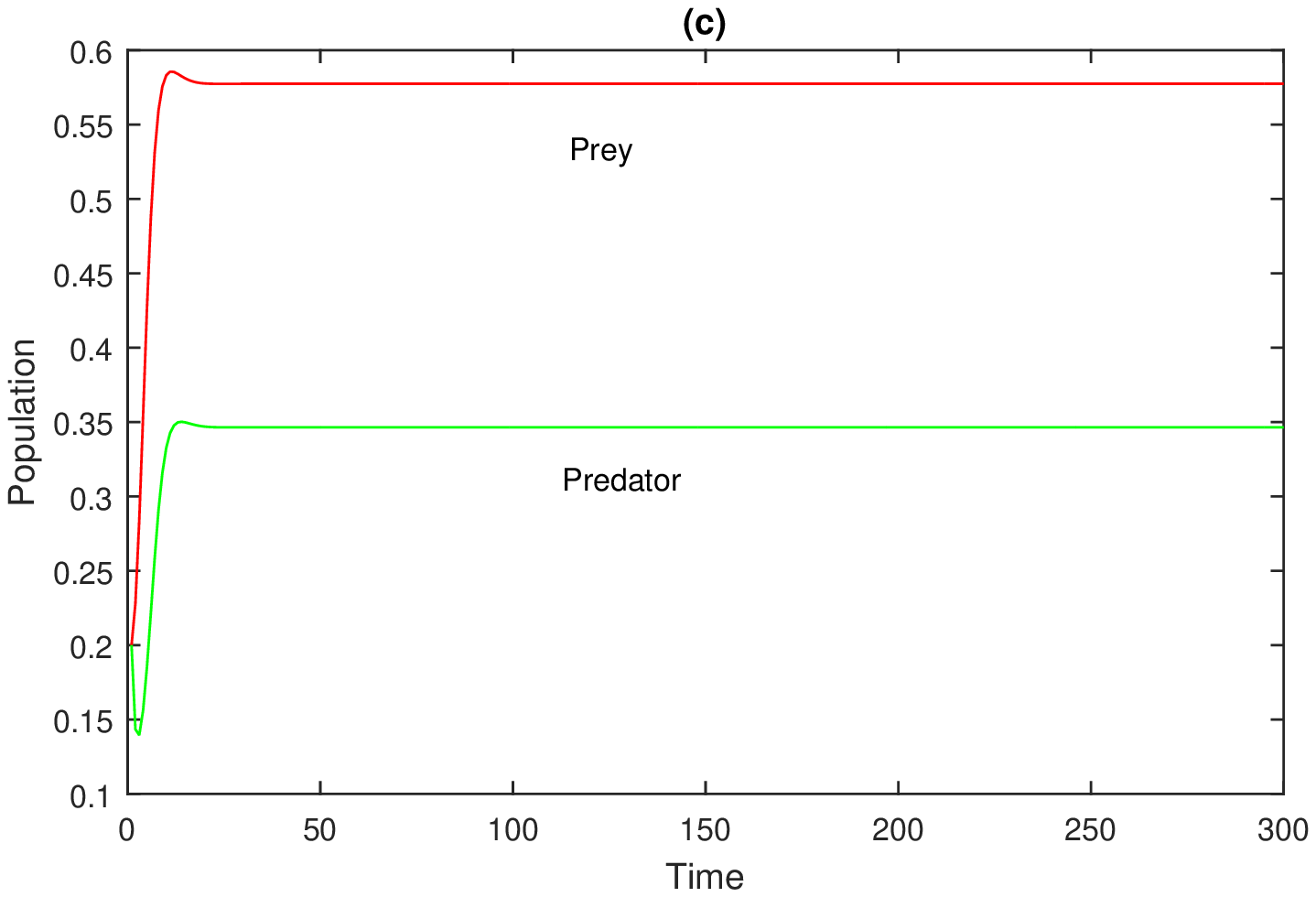}
	\includegraphics[width=2in, height=1.5in]{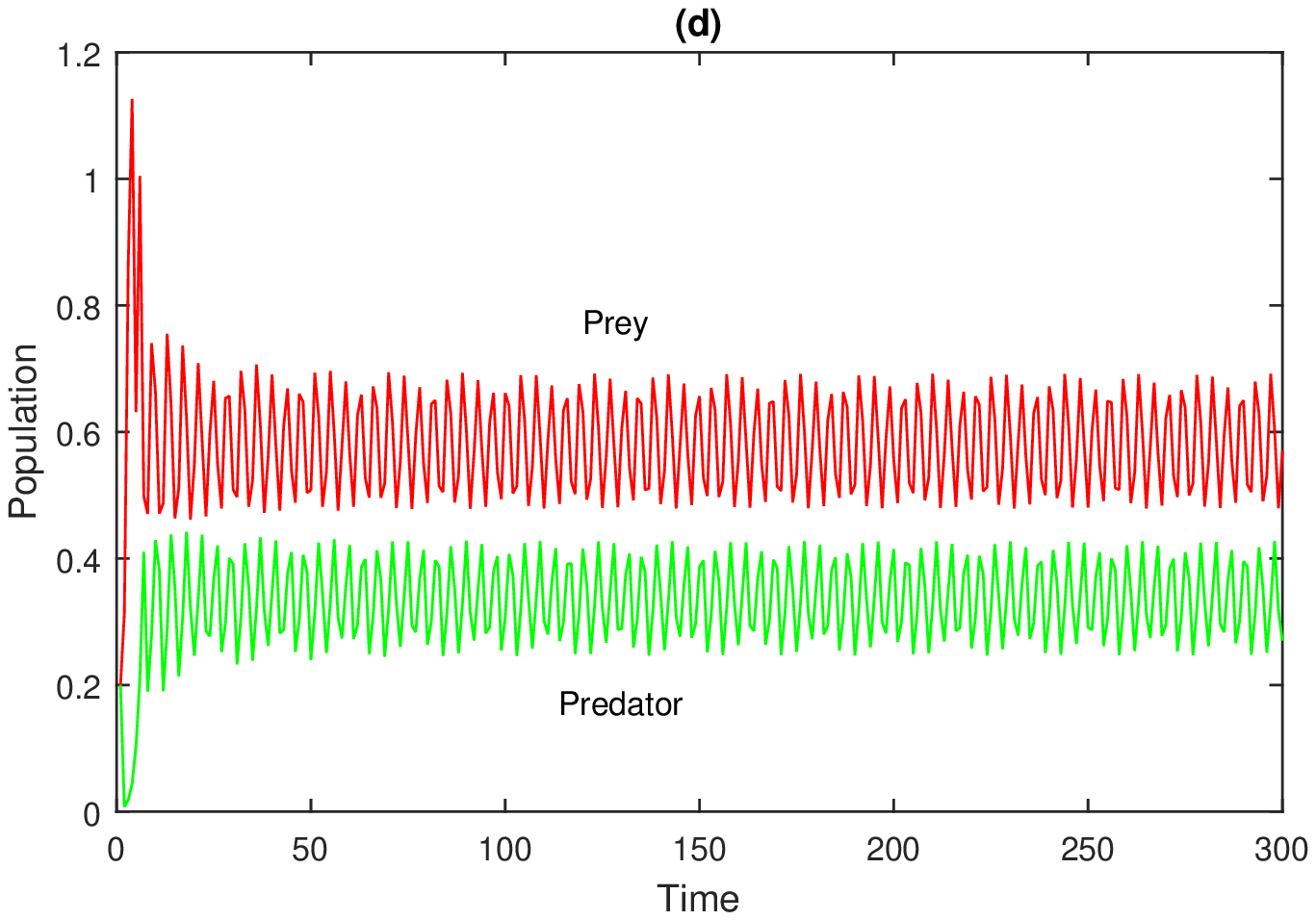}
\end{center}
{\bf Figure 5:} {\it Time series solutions of the NSFD system \eqref{model in discrete system} and Euler system \eqref{model in Euler system1} for two particular values of step-size (h). Here $h=2$ for Figs. (a) \& (b) and $h=2.67$ for Figs. (c) \& (d). Other parameters are in Fig. 4.}\\

In particular, we plot (Fig. 5) time series behavior of the NSFD system \eqref{model in discrete system} and Euler discrete system \eqref{model in Euler system1} for $h=2 (<min\{\frac{G}{H},\frac{2(1+\alpha \delta)^2}{G}\}=min\{2.6293, 2.4393\})$ and for $h=2.67 (>min\{2.6293, 2.4393\})$. The first two figures ($5a \& 5b$) show that both populations are stable when the step-size is $h=2$. Fig. 5c shows that populations of NSFD system \eqref{model in discrete system} remains stable for all step-size, indicating its dynamic consistency with the continuous system, but Fig. 5(d) shows that populations of Euler system \eqref{model in Euler system1} oscillate for $h=2.67$, indicating its dynamic inconsistency with its continuous counterpart.
\section{Summary}
Nonstandard finite difference (NSFD) scheme has gained lot of attentions in the last few years mostly for two reasons. First, it generally does not show spurious behavior as compared to other standard finite difference methods and second, dynamics of the NSFD model does not depend on the step-size. NSFD scheme also reduces the computational cost of traditional finite--difference schemes. In this work, we have studied two discrete systems constructed by NSFD scheme and forward Euler scheme of a well studied two--dimensional Holling-Tanner type predator--prey system with ratio-dependent functional response. We have shown that dynamics of the discrete system formulated by NSFD scheme are same as that of the continuous system. It preserves the local stability of the fixed point and the positivity of the solutions of the continuous system for any step size. Simulation experiments show that NSFD system always converge to the correct steady--state solutions for any arbitrary large value of the step size ($h$) in accordance with the theoretical results. However, the discrete model formulated by forward Euler method does not show dynamic consistency with its continuous counterpart. Rather it shows scheme--dependent instability when step--size restriction is violated.


\bibliographystyle{plain}

\end{document}